\documentclass[12pt]{amsart}
\usepackage{amsmath,amsfonts,amssymb,enumerate,mathrsfs,mathtools}

\makeatletter
\@namedef{subjclassname@2010}{
  \textup{2010} Mathematics Subject Classification}
\makeatother

\newtheorem{thm}{Theorem}[section]

\newtheorem{lem}[thm]{Lemma}
\newtheorem{prop}[thm]{Proposition}
\newtheorem{prob}[thm]{Problem}

\theoremstyle{definition}

\newtheorem{rem}[thm]{Remark}
\newtheorem{exa}[thm]{Example}

\numberwithin{equation}{section}

\newcommand{\N}{\mathbb{N}}                   
\newcommand{\R}{\mathbb{R}}                   
\newcommand{\C}{\mathbb{C}}                   
\newcommand{\A}{\mathscr{A}_a}                
\newcommand{\Na}{\mathcal{N}}                 
\newcommand{\Ro}{\mathcal{R}}                 
\newcommand{\vp}{\varphi}

\begin{document}


\baselineskip=17pt



\title[On solutions of linear equations]{On solutions of linear equations with polynomial coefficients}

\author[J. Adamus]{Janusz Adamus}
\address{Department of Mathematics, The University of Western Ontario, London, Ontario, Canada N6A 5B7}
\email{jadamus@uwo.ca}

\author[H. Seyedinejad]{Hadi Seyedinejad}
\address{Department of Mathematics, The University of Western Ontario, London, Ontario, Canada N6A 5B7}
\email{sseyedin@uwo.ca}
\thanks{J. Adamus's research was partially supported by the Natural Sciences and Engineering Research Council of Canada}

\date{}

\begin{abstract}
We show that a linear functional equation with polynomial coefficients need not admit an arc-analytic solution even if it admits a continuous semialgebraic one. We also show that such an equation need not admit a Nash regulous solution even if it admits an arc-analytic one.
\end{abstract}

\subjclass[2010]{Primary 14P10, 14P20, 14P99}

\keywords{arc-analytic function, Nash regulous function, semialgebraic geometry}

\maketitle


\section{Introduction}
\label{sec:intro}

The present note is concerned with existence of solutions to linear equations with polynomial coefficients in various classes of semialgebraic functions in $\R^n$. Recall that a set $X$ in $\R^n$ is called \emph{semialgebraic} if it can be written as a finite union of sets of the form $\{x\in\R^n: p(x)=0,q_1(x)>0,\dots,q_r(x)>0\}$, where $r\in\N$ and $p,q_1,\dots,q_r$ are polynomial functions. Given $X\subset\R^n$, a \emph{semialgebraic function} $f:X\to\R$ is one whose graph is a semialgebraic subset of $\R^{n+1}$.

A continuous function $f:\R^n\to\R$ is said to be \emph{regulous} if there exist polynomial functions $p$ and $q$ such that the zero locus of $q$ is nowhere dense in $\R^n$ and $f(x)=p(x)/q(x)$ whenever $q(x)\neq0$. A real analytic semialgebraic function on $\R^n$ is called \emph{Nash}. A continuous function  $f:\R^n\to\R$ is said to be \emph{Nash regulous} if there exist Nash functions $g$ and $h$ such that the zero locus of $h$ is nowhere dense in $\R^n$ and $f(x)=g(x)/h(x)$ whenever $h(x)\neq0$. Finally, recall that a function $f:X\to\R$ is called \emph{arc-analytic} if it is analytic along every arc, that is, $f\circ\gamma$ is analytic for every real analytic $\gamma:(-1,1)\to X$.
We shall denote the regulous, Nash regulous, and arc-analytic semialgebraic functions on $\R^n$ by $\Ro^0(\R^n)$, $\Na^0(\R^n)$ and $\A(\R^n)$, respectively. We have
\begin{equation}
\label{eq:rel}
\Ro^0(\R^n)\subset\Na^0(\R^n)\subset\A(\R^n)\,.
\end{equation}
The first inclusion is trivial and the second one follows from \cite[Prop.\,3.1]{Kuch}. Both inclusions are strict.

The above classes of semialgebraic functions have been extensively studied recently (see, e.g., \cite{AS1, AS2, FHMM, Kuch} and the references therein), in particular, in the context of the following problem of Fefferman and Koll{\'a}r \cite{FK}.
\medskip

Consider a linear equation
\begin{equation}
\label{eq:FK}
f_1\vp_1+\dots+f_r\vp_r=g,
\end{equation}
where $g$ and the $f_j$ are continuous (real-valued) functions on $\R^n$. Fefferman-Koll{\'a}r asked whether assuming that $g$ and the $f_j$ have some regularity properties, one could find a solution $(\vp_1,\dots,\vp_r)$ to \eqref{eq:FK} with similar regularity properties.

This is a difficult problem, even when the coefficients of \eqref{eq:FK} are polynomial.
One line of attack is to instead consider a somewhat easier question:

\begin{prob}
\label{prob:1}
Suppose that \eqref{eq:FK} admits a solution $(\vp_1,\dots,\vp_r)$ within some class of functions. Does there exist then a solution to \eqref{eq:FK} within a strictly smaller class?
\end{prob}

In the semialgebraic setting, the most general positive answer to this problem is given by \cite[Cor.\,29(1)]{FK}: \emph{If $f_1,\dots,f_r$ are polynomial, $g$ is semialgebraic and \eqref{eq:FK} admits a continuous solution, then it admits a continuous semialgebraic solution.}
In a similar vein, Kucharz and Kurdyka showed that, in case $n=2$, if $f_1,\dots,f_r,g$ are regulous then \eqref{eq:FK} admits a continuous solution if and only if it admits a regulous solution (cf. \cite[Cor.\,1.7]{KuKu}).

On the other hand, the above is known to fail for $n\geq3$. Namely, by \cite[Ex.6]{KN}, there exist $f_1,f_2,g\in\R[x,y,z]$ such that $f_1\vp_1+f_2\vp_2=g$ admits a continuous solution, but no regulous one.
Nonetheless, the solution from \cite[Ex.6]{KN} is Nash regulous, and in \cite{Kuch} Kucharz conjectured that existence of a continuous solution to \eqref{eq:FK} should imply the existence of a Nash regulous one, for any $n\geq1$, provided $f_1,\dots,f_r,g$ are polynomial.

The main goal of this note is to show that the latter is not the case. In Example~\ref{ex:no-arc-an}, we show that there exists a linear equation with polynomial coefficients which admits a continuous solution, but no arc-analytic one. By \eqref{eq:rel}, it follows that there is no Nash regulous solution either. Perhaps even more interestingly, in Example~\ref{ex:no-Nash-reg} we show a linear equation with polynomial coefficients that \emph{does} admit an arc-analytic solution and has no Nash regulous solution nonetheless. Both our examples are modifications of \cite[Ex.6]{KN}.
\medskip


\section{Toolbox}
\label{sec:tools}

The following facts will be needed in Examples~\ref{ex:no-arc-an} and~\ref{ex:no-Nash-reg}.

\begin{prop}
\label{prop:BM}
Let $f:\R^n\to\R$ be a semialgebraic function. Then, $f$ is arc-analytic if and only if there exists a mapping $\pi:\widetilde{R}\to\R^n$ which is a finite sequence of blowings-up with smooth algebraic centers, such that the composite $f\circ\pi$ is a Nash function.
\end{prop}

\begin{proof}
This is a special case of \cite[Thm.\,1.4]{BM}.
\end{proof}

Functions satisfying the conclusion of Proposition~\ref{prop:BM} are called \emph{blow-Nash}.

\begin{rem}
\label{rem:1}
A function $f:\R\to\R$ is arc-analytic if and only if it is real analytic.
This follows directly from the definition of arc-analytic functions.
\end{rem}

Recall that a Nash set (i.e., the zero set of a Nash function) in $\R^n$ is said to be \emph{Nash irreducible} if it cannot be realized as a union of two proper Nash subsets. A set is called \emph{Nash constructible} if it belongs to the Boolean algebra generated by the Nash subsets in $\R^n$.

\begin{rem}[{cf. \cite[Ex.\,2.3]{S}}]
\label{rem:2}
The graph $\Gamma_{\!f}$ of \,$f(x,y)=\sqrt{x^4+y^4}$ is not Nash constructible in $\R^3$.

Indeed, let $X\coloneqq\{(x,y,z)\in\R^3:z^2=x^4+y^4\}$. We claim that $X$ is Nash irreducible.
First, note that $z^2-x^4-y^4$ is an irreducible element in the ring of convergent power series over $\C$. This implies that the set $\{z^2-x^4-y^4=0\}\subset\C^3$ has an irreducible (complex analytic) germ at the origin, of (complex) dimension $2$. On the other hand, the (real analytic) germ of $X$ at the origin is of (real) dimension $2$. Hence, its complexification has to be given by precisely $\{z^2-x^4-y^4=0\}$. It follows that the germ $X_0$ is irreducible, and there is thus no way to decompose $X$ into proper analytic subsets. (See \cite{C} for details on real analytic germs and their complexifications.)

The irreducibility of $X$ implies that $X$ is the smallest Nash set in $\R^3$ containing $\Gamma_{\!f}$. Therefore, by \cite[Prop.\,2.1]{Kuch}, if $\Gamma_{\!f}$ were Nash constructible then it would need to contain the smooth locus of $X$. This is not the case, however, because $X$ contains also the graph of $g(x,y)=-\sqrt{x^4+y^4}$.
\end{rem}
\medskip

The following result is new, though it follows easily from \cite{Kuch}.

\begin{lem}
\label{lem:3}
Let $n\geq1$ and let $f,g\in\A(\R^n)$. If the zero locus of $g$ is nowhere-dense in $\R^n$ and the function $f/g$ extends continuously to $\R^n$, then this extension is in $\A(\R^n)$.
\end{lem}

\begin{proof}
By Proposition~\ref{prop:BM} above, there is a finite sequence $\pi:\widetilde{R}\to\R^n$ of blowings-up with smooth algebraic centers such that $f\circ\pi$ and $g\circ\pi$ are Nash functions on the Nash manifold $\widetilde{R}$. Continuity of $f/g$ implies that $(f\circ\pi)/(g\circ\pi):\widetilde{R}\to\R$ is a Nash regulous function. By \cite[Prop.\,3.1]{Kuch}, Nash regulous functions are arc-analytic, and hence there is a finite sequence $\sigma:\widehat{R}\to\widetilde{R}$ of blowings-up with smooth algebraic centers such that $(f/g)\circ\pi\circ\sigma=\dfrac{f\circ\pi}{g\circ\pi}\circ\sigma:\widehat{R}\to\R$ is Nash, by Proposition~\ref{prop:BM} again. Therefore, $f/g$ is arc-analytic.
\end{proof}


\section{Examples}
\label{sec:examples}

\begin{exa}
\label{ex:no-arc-an}
Consider the equation
\begin{equation}
\label{eq:H}
x^3y\,\vp_1+(x^3-y^3z)\vp_2=x^4.
\end{equation}
We claim that
\[
\vp_1(x,y,z)=z^{1/3},\qquad\vp_2(x,y,z)=\frac{x^3}{x^2+xyz^{1/3}+y^2z^{2/3}}
\]
is a continuous solution to \eqref{eq:H}, but no semialgebraic arc-analytic solution exists.
The function $\vp_1$ is clearly continuous. To see that $\vp_2$ is continuous, first note that the set
\[
\{(x,y,z)\in\R^3:x^2+xyz^{1/3}+y^2z^{2/3}=0\}
\]
is the union of the $y$-axis and the $z$-axis. Therefore, $x\to0$ whenever $(x,y,z)$ approaches the locus of indeterminacy of $\vp_2$. On the other hand, we have
\[
x^2+xyz^{1/3}+y^2z^{2/3}\geq\frac{1}{2}\left(x^2+y^2z^{2/3}\right),
\]
which shows that $\dfrac{x^2}{x^2+xyz^{1/3}+y^2z^{2/3}}\,$ is bounded. Hence, $\vp_2$ can be continuously extended by zero to $\R^3$.

Suppose now that \eqref{eq:H} has an arc-analytic solution $(\psi_1,\psi_2)$. Set
\[
S\coloneqq\{(x,y,z)\in\R^3:x^3=y^3z\},
\]
and note that $y$ vanishes on $S$ only when $x$ does so. Therefore, $x/y$ is a well defined function on $S\setminus\{x=0\}$, and thus, by \eqref{eq:H}, we obtain that
\[
\left.\psi_1\right|_{S\setminus\{x=0\}}=\left.\frac{x}{y}\right|_{S\setminus\{x=0\}}\,.
\]
Note that every point $(0,0,c)$ of the $z$-axis can be approached within $S\setminus\{x=0\}$, even by an analytic arc. Indeed, for instance, by the arc $(\sqrt[3]{c}t,t,c)$ for $c\neq0$ and the arc $(t^2,t,t^3)$ for $c=0$. This allows us to write
\[
\lim_{(x,y,z)\to(0,0,c)}\psi_1(x,y,z)=\lim_{(x,y,z)\to(0,0,c)}\left.\frac{x}{y}\right|_{S\setminus\{x=0\}}\!\!=\,c^{1/3}\,.
\]
Therefore, $\psi_1|_{z\textrm{-axis}}=z^{1/3}$, by continuity. This contradicts the arc-analyticity of $\psi_1$, by Remark~\ref{rem:1}.
\qed
\end{exa}

\begin{exa}
\label{ex:no-Nash-reg}
Consider now the equation
\begin{equation}
\label{eq:J}
x^4y^2\,\vp_1+(x^4-y^4(z^4+w^4))\,\vp_2=x^6.
\end{equation}
We claim that
\[
\vp_1=\sqrt{z^4+w^4},\qquad\vp_2=\frac{x^4}{x^2+y^2\sqrt{z^4+w^4}}
\]
is an arc-analytic solution to \eqref{eq:J}, but no Nash regulous solution exists.
It is easy to see that the function $\sqrt{z^4+w^4}$ is blow-Nash, and hence arc-analytic, by Proposition~\ref{prop:BM}.
Thus, by Lemma~\ref{lem:3}, to see that $\vp_2$ is arc-analytic, it suffices to show that it extends continuously to $\R^4$.
First, note that the set
\[
\{(x,y,z,w)\in\R^4:x^2+y^2\sqrt{z^4+w^4}=0\}
\]
is the union of the $y$-axis and the $(z,w)$-plane.
Therefore, $x\to0$ whenever $(x,y,z,w)$ approaches the locus of indeterminacy of $\vp_2$.
On the other hand, the function $\dfrac{x^2}{x^2+y^2\sqrt{z^4+w^4}}\,$ is clearly bounded. Hence, $\vp_2$ can be continuously extended by zero to $\R^4$.

Suppose now that \eqref{eq:J} has a Nash regulous solution $(\psi_1,\psi_2)$. Set
\[
S\coloneqq\{(x,y,z,w)\in\R^4:x^4=y^4(z^4+w^4)\},
\]
and note that $y$ vanishes on $S$ only when $x$ does so. Therefore, $(x/y)^2$ is a well defined function on $S\setminus\{x=0\}$, and thus, by \eqref{eq:J}, we obtain that
\[
\left.\psi_1\right|_{S\setminus\{x=0\}}=\left.\frac{x^2}{y^2}\right|_{S\setminus\{x=0\}}\,.
\]
Note that the $(z,w)$-plane is contained in $S$, and every point $(0,0,c,d)$ of the $(z,w)$-plane can be approached within $S\setminus\{x=0\}$, even by an analytic arc. Indeed, for instance, by the arc $(\sqrt[4]{c^4+d^4}t,t,c,d)$ for $c^4+d^4\neq0$ and the arc $(\sqrt[4]{2}t^2,t,t,t)$ for $c^4+d^4=0$. This allows us to write
\[
\lim_{(x,\dots,w)\to(0,0,c,d)}\psi_1(x,y,z,w)=\lim_{(x,\dots,w)\to(0,0,c,d)}\left.\frac{x^2}{y^2}\right|_{S\setminus\{x=0\}}=\sqrt{c^4+d^4}\,.
\]
Therefore, $\psi_1|_{(z,w)\textrm{-plane}}=\sqrt{z^4+w^4}$, by continuity. This is impossible for a Nash regulous function though, because by \cite[Cor.\,3.2]{Kuch} the graph of a Nash regulous function (and hence its intersection with any coordinate plane) is a closed Nash constructible set. However, the graph of $f(z,w)=\sqrt{z^4+w^4}$ is not Nash constructible, by Remark~\ref{rem:2}.
\qed
\end{exa}


\end{document}